\documentclass{amsart}

\usepackage[latin1]{inputenc}
\usepackage{amssymb}
\usepackage{amsfonts}
\usepackage{amsmath}
\usepackage{url}
\usepackage{latexsym,longtable,amscd,booktabs,array}
\usepackage{color}

\theoremstyle{plain}
\newtheorem{theorem}{Theorem}
\newtheorem{lemma}{Lemma}

\newtheorem{corollary}{Corollary}

\newtheorem{definition}{Definition}
\newtheorem{proposition}{Proposition}

\newtheorem{remark}{Remark}

\newtheorem{example}{Example}

\newcommand{\C}{\mathbb{C}}

\newcommand{\Q}{\mathbb{Q}}

\newcommand{\Z}{\mathbb{Z}}

\newcommand{\Qbar}{\overline{\Q}}

\newcommand{\aut}[1]{{\rm Aut}( {#1} ) }
\newcommand{\gal}[1]{{\rm Gal}( {#1} ) }

\newcommand{\ord}[1]{{\rm ord}\,#1}

\newcommand{\calD}{{\mathcal{D}}}

\newcommand{\Qisog }{\stackrel{\Q}{\sim}}
\newcommand{\jac}[1]{{\rm Jac}( {#1} ) }

\def\QED{\hfill $\Box$}

\newcommand{\twoheadlongrightarrow}{\relbar\joinrel\twoheadrightarrow}
\newcolumntype{L}[1]{>{\rightskip=0pt plus1fil}p{#1\textwidth}<{}}
\newcolumntype{Z}{L{0.14}@{:\,\,\,\,}L{0.78}}

\begin{document}

\title{Non-hyperelliptic modular curves of genus $3$}
\author{Enrique Gonz{\'a}lez-Jim{\'e}nez}
\address{Universidad Aut{\'o}noma de Madrid, Departamento de Matem{\'a}ticas and Instituto de Ciencias Matemáticas (CSIC-UAM-UC3M-UCM), Madrid, Spain}
\email{enrique.gonzalez.jimenez@uam.es}

\author{Roger Oyono}
\address{\'Equipe GAATI, Universit\'e de la Polyn\'esie Fran\c{c}aise, 
BP 6570 - 98702 Faa'a - Tahiti - French Polynesia}
\email{roger.oyono@upf.pf}
\thanks{The first author was supported in part by grant MTM 2006-10548 and CCG06-UAM/ESP-0477.}

\date{6 December 2009}

\begin{abstract}
A curve $C$ defined over $\Q$ is modular of level $N$  if there exists a non-constant morphism from $X_1(N)$ onto $C$ defined over $\Q$ for some positive integer $N$. We provide a sufficient and necessary condition for the existence of a modular non-hyperelliptic curve $C$ of genus $3$ and level $N$ such that $\jac{C}$ is $\Q$-isogenous to a given three dimensional $\Q$-quotient  of $J_1 (N)$. Using this criterion, we present an algorithm to compute explicitly equations for modular non-hyperelliptic curves of genus $3$.  Let $C$ be a modular curve of level $N$, we say that $C$ is new if the corresponding morphism between $J_1(N)$ and $\jac{C}$ factors through the new part of $J_1(N)$. 
We compute equations of  $44$ non-hyperelliptic new modular curves of genus $3$, that we conjecture to be the complete list of this kind of curves. Furthermore, we describe some aspects of non-new modular curves and we present some examples that show the ambiguity of the non-new modular case.
\end{abstract}

\maketitle

 

\section{Introduction}\label{intro}
Let $N$ be a positive integer and $X_1(N)$ (resp. $ X_0(N)$) be the classical modular curve corresponding to the modular group $\Gamma_1(N)$ (resp. $\Gamma_0(N)$). Many papers have already been devoted to the problem of  finding $\Q$-rational models for these modular curves and their quotients \cite{MShimura95,J_Gon, murabayashi92, MR1472817, hasegawa:hashimoto96,MR1692024}. In this work we are interested in modular curves defined over $\Q$ which are dominated over $\Q$ by $X_1(N)$.

In \cite{GonGon, BakGonPoo} the concept of new modular curve is introduced. These are curves dominated by $X_1(N)$ such that the corresponding morphism on their jacobians factors through the new part of the jacobian of $X_1(N)$. For the genus $1$ case, the concept of new modular curve and modular curve are equivalent. Shimura \cite{Shimura71} proved that any elliptic curve with complex multiplication is modular, and provided in this way the first infinite set of new modular curves of genus $1$. Furthermore, in a series of papers, Wiles et al. \cite{Wiles,TWiles,MR1839918} proved that every elliptic curve defined over $\Q$ is modular and thus new modular. Conversely, all new modular curves of genus $1$ are elliptic curves defined over $\Q$.  In contrast to the new modular elliptic curve case, for a fixed genus $g\geq 2$ the set of new modular curves of genus $g$ (up to $\Q$-isomorphism) is finite and computable \cite{BakGonPoo}. In the genus $2$ case,   \cite{BakGonPoo,GonGon} provide a complete list of new modular curves. More precisely, in their proof \cite{BakGonPoo} of computability of new modular curves of fixed genus $g>1$, the authors develop a deterministic method that provides a finite (but enormous) list containing, amongst others, all the new modular curves of genus $g$. The large amount of curves appearing in this  list makes the computation of all new modular curves of genus $g>2$ impossible nowadays. For example, for $g=3,4,5,6$  there are  respectively $10^{105}, 10^{239},10^{455},10^{844}$ possibilities.
Moreover,  \cite{BakGonPoo} provides a sufficient and necessary condition to verify if a new modular abelian variety is $\Q$-isogenous to the jacobian of a modular hyperelliptic curve. That is, for each level $N$, they provide a method to compute  all the new modular hyperelliptic curves defined over $\Q$ of level $N$. For the non-hyperelliptic new modular case, they provide a necessary, but not sufficient condition based on the canonical embedding (cf. Remark \ref{obs1}). 

The aim of this paper is to study the simplest case of non-hyperelliptic new modular curves, i.e.\, the case of genus $3$ (smooth plane quartics). We first provide a necessary and sufficient condition for a non-hyperelliptic curve to be modular of level $N$ with the additional requirement that its holomorphic differentials correspond to the holomorphic differentials of a given modular abelian $3$-fold defined over $\Q$. We then restrict our attention to the computation of all non-hyperelliptic new modular curves of genus $3$ up to a fixed level (see Appendix). 

This paper is organized as follows: In Sections 2 and 3 we review the necessary technical background about modular curves and non-hyperelliptic genus $3$ curves respectively. In Section 4 we present a method that allows us to recognize if a modular abelian $3$-fold corresponds to the jacobian of a non-hyperelliptic modular curve of genus $3$.  We apply this method to compute all the new modular non-hyperelliptic curves of genus $3$  up to certain levels. 
 In Section 5 we present some examples that show the ambiguity of the non-new modular case. We conclude this paper with an appendix that gives equations of  $44$ non-hyperelliptic new modular curves of genus $3$, and we expect these to be the complete list of this kind of curves. 

{\it Remark:} 
A different approach to the one studied in this paper is the one researched at \cite{oyono}. There the author developes an algorithm to recognize if a modular abelian $3$-fold is the jacobian of a non-hyperelliptic curve of genus $3$. Note that in our problem, if a curve is modular then its jacobian is modular. Nevertheless, the converse is not true in general.

{\it Notation:} 
All curves and varieties in this paper are smooth and projective, and all the fields will be of characteristic zero. If $X$ is a variety over a field $K$, $\Omega^1=\Omega^1_{X/K}$ denotes the sheaf of holomorphic $1$-forms. If $A$ and $B$ are two abelian varieties defined over a field $K$, the notation $A\stackrel{K}{\sim}B$ means that $A$ and $B$ are $K$-isogenous. Let $k,N\in \mathbb{N}$, we denote by $S_k(N)$ the vector space of cuspidal forms of weight $k$ for the modular subgroup $\Gamma_1(N)$. Throughout the paper all the modular curves and abelian varieties are defined over $\Q$ and we will use the labelling of modular forms and abelian varieties as it was introduced in \cite[Appendix]{BakGonPoo}. For the sake of completeness, we remind this labelling at the Appendix. Along the paper we will use the canonical identification between the spaces $H^0(C,\Omega^1)$ and $H^0(\jac{C},\Omega^1)$ (cf. \cite[\S 2.9, Prop. 8]{shimura-taniyama}).


\section{Modular curves}\label{section2}

This section is dedicated to the basic notions about modularity that will be used in the rest of the paper. \cite[\S 3.1]{BakGonPoo}  is a good reference where all the necessary background is included.

\begin{definition}
An abelian variety $A$ over $\Q$ is said to be {\it modular of level $N$} if there exists a surjective $\Q$-morphism $\nu:J_1(N)\twoheadlongrightarrow A$. In that case, we say that A is {\it new (of level $N$)} if $\nu$ factorizes through $J_1^{\rm new}(N)$ over $\Q$. 
\end{definition}

Thanks to the knowledge of the decomposition of $J_1(N)$ over $\Q$, we have that if $A$ is a modular abelian variety then there exist $k$ normalized eigenforms $ f_1,\dots ,f_k$ in $S_2(N,\varepsilon_k)$ (where $S_2(N,\varepsilon)$ denotes the vector space of cuspidal forms of weight $2$, level $N$ and Nebentypus the Dirichlet character $\varepsilon$) such that 
\begin{equation}\label{eq_A}
A\stackrel{\Q}{\sim}A_{f_1}^{n_1}\times\dots\times A_{f_k}^{n_k}\,,
\end{equation}
where $A_{f_i}$ is the abelian variety defined over $\Q$ attached to $f_i$ by Shimura. These modular abelian varieties are very special. For instance, if $f$ is a newform with $q$-expansion  $f(q)=q+\sum_{n\ge 2}{a_n q^n}$ (here $q=e^{2\pi iz}$), then $A_f$ is $\Q$-simple and $K_f=\Q(\{a_n\})$ is a number field of degree equal to the dimension of $A_f$. 

For a modular abelian variety $\nu:J_1(N)\twoheadlongrightarrow A$ we define $S_2(\nu,A)$ to be the subspace of $H^0(J_1(N),\Omega^1)$ determined by the equality 
$$
\nu^{*} H^0(A,\Omega^1)=S_2(\nu,A)\frac{dq}{q} \,.
$$
If $A$ is new, it admits a unique quotient map from $J_1(N)$, and hence, the space $S_2(\nu,A)$ is canonically attached to $A$ and it is denoted simply by $S_2(A)$. Note that if $A=A_f$ then $S_2(A_f)=\langle {}^\sigma f(q) : \sigma\in \gal{\Qbar/\Q}\rangle$. On the other hand, if $A$ is non-new of level $N$, it may be presented in many ways as a quotient of $J_1(N)$, and thus this space depends on the quotient map $\nu$. Along the rest of the paper, whenever we refer to a modular abelian variety, we will assume implicitly that we are given a fixed map $\nu:J_1(N)\twoheadlongrightarrow A$, and thus we will write shortly $S_2(A)$ instead of $S_2(\nu,A)$.

\begin{definition}
A non-singular curve $C$ defined over $\Q$ is said to be {\it modular of level $N$} if there exists a non-constant $\Q$-morphism $\pi :X_1(N)\twoheadlongrightarrow C$. The modular curve $C$ is then said to be {\it new of level $N$} if its jacobian $\jac{C}$ is new of level $N.$
\end{definition}

Let $\pi :X_1(N)\twoheadlongrightarrow C$ be a modular curve, then we will use $S_2(C)$ to denote $S_2(\pi_*,\jac{C})$. Here the dominant $\Q$-morphism is uniquely determined by $\pi$, that is, the dominant morphism is $\pi_*:J_1(N)\twoheadlongrightarrow \jac{C}$.

As a first step to understand the structure of new modular curves, the authors of \cite{GonGon} showed that the set of new modular curves of genus $2$ over $\Q$ is finite, and that there are exactly 149 such curves whose jacobian  is $\Q$-simple. In \cite{BakGonPoo} the case of genus $2$ is completed, showing in particular that there are exactly $213$ new modular curves of genus $2$. Furthermore,  \cite{BakGonPoo} generalized the above approach for new modular curves with fixed genus $g\geq 2$:
 
\begin{theorem}\cite{BakGonPoo}
For each integer $g\geq 2$, the set of new modular curves over $\Q$ of genus $g$ is finite and computable.
\end{theorem}
\section{Non-hyperelliptic curves of genus $3$}

Since parts of the theory concerning non-hyperelliptic genus $3$ curves are not easily available in modern publications, we include here the theory we use throughout the article.  Specifically, we present some facts about holomorphic differentials, canonical embedding and automorphism subgroups. Some useful references can be found in \cite{accola_70,brieskorn,farkas_kra,komiya_kuri}.

Let $C$ be a non-hyperelliptic curve of genus $3$ defined over a field $K$ and let $\left\{ \omega_1, \omega_2, \omega_3\right\}$ be a basis of the space $H^0(C,\Omega^1)$ of holomorphic differential forms  on $C.$ The canonical embedding of $C$ with respect to this basis is given by
\[
\begin{array}{cccl}
	\phi  : & C  & \longrightarrow & \mathbb{P}^{2}\\
	&  P& \longmapsto & \phi(P):=[\omega_1(P): \omega_2(P): \omega_3 (P)]\, ,
\end{array}
\]
where $\omega (P)=g(P)$ for any expression $\omega = gdt_P,$ with $g,t_P\in K(C)$ and $t_P$ a local parameter at $P.$ The image $\phi (C)$ of $C$ by such a  canonical embedding $\phi$ is a smooth plane quartic defined over $K$ and, conversely, any smooth plane quartic is the image by a canonical embedding of a genus $3$ non-hyperelliptic curve.

From now on, let $C$ be a smooth plane quartic defined over $K$ by an affine model $f(x,y)=0$. In these coordinates there is a canonical basis of $H^0 (C,\Omega^1)$ given by
\begin{equation}\label{eq_hol_afin}
\frac{x}{\frac{\partial f}{\partial y}(x,y)}\,dx\, , \, \frac{y}{\frac{\partial f}{\partial y}(x,y)}\,dx \, , \, \frac{1}{\frac{\partial f}{\partial y}(x,y)}\,dx\,.
\end{equation}

The next result will be useful in Section \ref{auto_new}.

\begin{proposition}\label{quotg1}
Let $C$ be a  non-hyperelliptic curve of genus $3$ and $G$ an abelian subgroup of $\aut{C}$ such that the genus of $C/G$ is $1$ . Then $G$ is cyclic of order $2,3$ or $4$.
\end{proposition}

We need a technical lemma to prove this proposition.

\begin{lemma}\label{lem}
Let $C$ be a  non-hyperelliptic curve of genus $3$ and $G$ a cyclic subgroup of $\aut{C}$. Then the genus of $C/G$ is $0$ or $1$. Furthermore, if $|G|>4$ then the genus of $C/G$ is $0$.
\end{lemma}

{\bf Proof Lemma \ref{lem}:}  We are going to apply Hurwitz's formula to the covering $C\twoheadlongrightarrow C/G$ to obtain the desired result. First we obtain that the genus of $C/G$ is $0$, $1$ or $2$. Let us suppose that the genus is $2$, then by Hurwitz's formula we have that $|G|=2$. Therefore, there exists an involution of $C$ that is not bielliptic, and thus by  \cite[Corollary, p.283]{komiya_kuri} $C$ is hyperelliptic, in contradiction with the hypothesis. 

Again, using Hurwitz's formula we obtain that the genus of $C/G$ is $0$ if $|G|>4$. \QED 

{\bf Proof Proposition \ref{quotg1}:} According to the possible full automorphism group for $C$  \cite{komiya_kuri}, we have 
$$
|\aut{C}|\in\{1,2,3,2^2,2\cdot 3,7,2^3,3^2,2^4,2^3\cdot 3,2^4\cdot 3,2^5 \cdot 3 ,2^3\cdot  3 \cdot  7\}.
$$
We can then assume that there exist non-negative integers $a \leq 5, b \leq 2 $ and $c \leq 1$ such that $|G|= 2^a \cdot 3^b \cdot 7^c$. Then we have
\begin{itemize}
\item $c=0$. Since for $c\neq 0$ there would exist a cyclic subgroup $H<G$ of order $7$ and therefore by Lemma \ref{lem} the genus of $C/H$ should be $0$. Hence the genus of $C/G$ should also be $0$.
\item $b\leq 1$. The only possibility for $3^2$ to divide $|G|$ is that $G \cong \Z /9\Z$ (see the complete classification of the full automorphism group \cite{komiya_kuri}).  Then by Lemma \ref{lem} we have that the genus of $C/G$ is $0$, in contradiction with the hypothesis. 
\item $G\not \cong (\Z /2\Z)^2$. This is a direct conclusion using Accola's Theorem (\cite{accola_70},\cite[V.1.10]{farkas_kra}) and the fact that any involution of a non-hyperelliptic curve of genus $3$ is bielliptic (\cite[Corollary, p.283]{komiya_kuri}).
\item $|G|\neq 2^3,  2^4, 2^5$,  since any such abelian group has $(\Z /2\Z)^2$ or a cyclic subgroup of order $>4$ as a subgroup.
\item $|G|\neq 2 \cdot 3, 2^2 \cdot 3, 2^3 \cdot 3, 2^4 \cdot 3, 2^5 \cdot 3, \, $ since the corresponding groups have at least a cyclic subgroup of order $>4$. 
\end{itemize}
Combining the above statements yields the conclusion that $G$ is cyclic of order $2,3$ or $4$. \QED

\section{Non-hyperelliptic modular curves of genus 3}\label{section3}
\subsection{Modular Criteria}
For non-hyperelliptic curves of genus $3,$ Proposition \ref{nonhyper3}  will provide us with an effective criteria to determine when a $\Q$-factor of $J_1(N)$ is $\Q$-isogenous to the jacobian of a non-hyperelliptic modular curve of genus $3.$ There is a similar version of Proposition \ref{nonhyper3} for hyperelliptic curves \cite{BakGonPoo}, where the main difference  appears in condition (iii)(b) below which is not necessary for hyperelliptic curves. Lemma \ref{Lemmaalgo} will give extra information in the new case.

\begin{proposition}\label{nonhyper3}
Let $J_1(N)\twoheadlongrightarrow A$ be a modular abelian $3$-fold  defined over $\Q$. The following statements are equivalent:
\begin{enumerate}
\item[(i)] There exist a non-hyperelliptic genus $3$ curve $C$ and a non-constant $\Q$-morphism $\pi : X_1(N)\longrightarrow C$ such that $S_2(C)=S_2(A)$. In particular, $\jac{C}\Qisog A$.
\item[(ii)] There exist a non-hyperelliptic genus $3$ curve $C'_{/\C}$ and a non-constant $\C$-morphism $\pi' : X_1(N)\longrightarrow C'$ such that $S_2(C')=S_2(A)$.
\item[(iii)] For every basis $\left\{ f_1,f_2,f_3\right\}$ for $S_2(A)$, there exists an irreducible and non-singular homogeneous polynomial $ F(X,Y,Z)\in\C[X,Y,Z]$ of degree $4$ such that:
\begin{itemize}
\item[(a)] $F(f_1,f_2,f_3) =0$, 
\item[(b)] Denote by  $f_i'$  the derivative of $f_i$ with respect to the complex variable $z\in\mathbb{H}$, for $i=1,3$, and  define the function 
$$
\displaystyle \psi_F(f_1,f_2,f_3):=\frac{f_3 f_1' - f_1 f_3'}{\frac{\partial F}{\partial Y}(f_1, f_2,f_3)}\in \C(X_1(N)).
$$
Then $\displaystyle \psi_F(f_1,f_2,f_3)$ is a constant function.
\end{itemize}
\end{enumerate}
\end{proposition}

The following technical lemma will be used in the proof of the previous proposition.

\begin{lemma}\label{ritz}
Let $\pi :C_1\longrightarrow C_2$ be a non-constant morphism between curves defined over a field $K$ of characteristic zero. Let $\{\omega_1,\dots,\omega_g\}$ be a basis of $H^0(C_2,\Omega^1)$ and  $f\in K(C_1)$ such that $f \pi^*\omega_i$ belongs to $\pi^*H^0(C_2,\Omega^1)$, for all $i$. Then $f\in K$.
\end{lemma}

{\bf Proof:} One has $f \pi^*\omega_i=\pi^*\mu_i$ for all $i$, where $\mu_i\in H^0(C_2,\Omega^1)$. Applying the pushdown of $\pi$ we obtain $\pi_*(f \pi^*\omega_i)=\pi_*(\pi^*\mu_i)$. Therefore
$$
\pi_*({\rm div}(f))+{\rm deg}\,\pi\,{\rm div}(\omega_i)={\rm deg}\,\pi\,{\rm div}(\mu_i).
$$
Let ${\rm div}(f)=\sum_{P\in C_1}{n_P P}$, then $\pi_*({\rm div}(f))=\sum_{P\in C_1}{n_P \pi(P)}$. Assume that $f$ is not a constant, then there exists $P_0\in C_1$ such that $n_{P_0}<0$. Let $Q_0=\pi(P_0)$. Thus  $n_{P_0}+{\rm deg}\,\pi\, {\rm ord}_{Q_0}(\omega_i)= {\rm deg}\,\pi\, {\rm ord}_{Q_0}(\mu_i)\ge 0$. Therefore ${\rm ord}_{Q_0}(\omega_i)>0$ for all $i$. This is not possible because the linear system of regular differentials is base-point free.
\QED



{\bf Proof Proposition \ref{nonhyper3}:} The assertion that (i) implies (ii) is trivial. Assuming (ii) we are going to prove (iii). Let
$$
\mathbb{H}\stackrel{\rho}{\longrightarrow} X_1(N) \stackrel{\pi'}{\longrightarrow} C' \stackrel{\phi}{\longrightarrow}\mathbb{P}^2=\mathbb{P}(H^0(C'_{/\C},\Omega^1))\, .
$$
Then let $x_1,x_2,x_3$ the coordinates  on $\mathbb{P}^2$ and
$$
\begin{array}{l}
\omega_i=\phi^*(x_i) \in H^0(C'_{/\C},\Omega^1), \,\,\,i=1,2,3\, ,\\
\mu_i=\pi'^*(\omega_i) \in H^0(X_1(N)_{/\C},\Omega^1), \,\,\,i=1,2,3.
\end{array}
$$
Since $C'$ is non-hyperelliptic there exists an irreducible and non-singular homogeneous polynomial $F\in \C[X,Y,Z]$ of degree $4$ such that $\phi(C'):F(x_1,x_2,x_3)=0$. Hence $F((\phi\circ\pi')^*(x_1),(\phi\circ\pi')^*(x_2),(\phi\circ\pi')^*(x_3))=0$, i.e. $F(\mu_1,\mu_2,\mu_3)=0$. Let us call $x=\mu_1/\mu_3,y=\mu_2/\mu_3\in\C(X_1(N))$ and $f\in \C[X_1,X_2]$ the dehomogenization of $F$, then $f(x,y)=0$. Now we identify $C'$ with $\phi(C')\subset \mathbb{P}^2$, and let $f(u,v)$ the affine equation. Then by (\ref{eq_hol_afin}) we have 
$$
H^0 (C',\Omega^1) = \langle \frac{u}{\frac{\partial f}{\partial X_2}(u,v)}\,du\, , \frac{v}{\frac{\partial f}{\partial X_2}(u,v)}\,du \, , \frac{1}{\frac{\partial f}{\partial X_2}(u,v)}\,du\rangle_\C\,. 
$$
Therefore 
$$
\begin{array}{l}
\pi'^*(u)=\pi'^*\left(\frac{\omega_1}{\omega_3}\right)=\frac{\pi'^*(\omega_1)}{\pi'^*(\omega_3)}=\frac{\mu_1}{\mu_3}=x,\\
\pi'^*(v)=\pi'^*\left(\frac{\omega_2}{\omega_3}\right)=\frac{\pi'^*(\omega_2)}{\pi'^*(\omega_3)} =\frac{\mu_2}{\mu_3}=y.
\end{array}
$$
Thus
$$
\pi'^*H^0 (C',\Omega^1) = \langle \frac{x}{\frac{\partial f}{\partial X_2}(x,y)}\,dx\, , \frac{y}{\frac{\partial f}{\partial X_2}(x,y)}\,dx \, , \frac{1}{\frac{\partial f}{\partial X_2}(x,y)}\,dx\rangle_\C\,.
$$
Classically, one can identify functions (resp. regular differentials) on $X_1(N)$ with modular functions (resp. modular forms of weight $2$) with respect to $\Gamma_1(N)$. So we let $\mu_i=f_i(z)\, dz$ with $f_i\in S_2(N)$, for $i=1,2, 3$. So $x=f_1/f_3$ and $y=f_2/f_3$. In particular,  $\{f_1,f_2,f_3\}$ form a basis for $S_2(A)$, $F(f_1,f_2,f_3)=0$ and 
$$
\begin{array}{l}
\frac{x}{\frac{\partial f}{\partial X_2}(x,y)}\,dx=\psi_F(f_1,f_2,f_3) f_1(z)dz\\
\frac{y}{\frac{\partial f}{\partial X_2}(x,y)}\,dx=\psi_F(f_1,f_2,f_3) f_2(z)dz\\
\frac{1}{\frac{\partial f}{\partial X_2}(x,y)}\,dx=\psi_F(f_1,f_2,f_3) f_3(z)dz
\end{array}
$$
Now let us proof that $\psi_F(f_1,f_2,f_3)$ is a modular function. A straightforward computation shows that, if $f_1,f_3\in S_2(N)$,  then $f'_1f_3 - f_1f'_3\in S_6(N)$. The modular form $\frac{\partial F}{\partial Y}(f_1,f_2,f_3)$ belongs to $S_6(N)$ since $\frac{\partial F}{\partial Y}(X,Y,Z)$ is a homogeneous polynomial of degree $3$ . Therefore 
 $$
 \frac{f'_1f_3 - f_1f'_3}{\frac{\partial F}{\partial X_2}(f_1,f_2,f_3)}\in \C(X_1(N)). 
 $$
Let us call it $\Psi(z)$. Then for all $i$ we have
$$
\psi_F(f_1,f_2,f_3) f_i(z)dz=\Psi \cdot \mu_i= \Psi\cdot  \pi' (\omega_i)\in  \pi'^*H^0 (C',\Omega^1).
$$
Then applying Lemma \ref{ritz} we conclude that the function $\Psi(z)$ is constant.

Let now $\left\{ g_1,g_2,g_3\right\}$ be another basis for $S_2(A)$ and $M\in {\rm GL}_3 (\C)$ from the matrix of the change of basis $\left\{ g_1,g_2,g_3\right\}$ to $\left\{ f_1,f_2,f_3\right\}$, that is $(g_1,g_2,g_3)^t=M^{-1}(f_1,f_2,f_3)^t$. Lets define
$$
G(X,Y,Z)=F((X,Y,Z)\cdot M^t)\in \C[X,Y,Z].
$$
Then $G$ is an irreducible and non-singular homogeneous polynomial of degree $4$ such that $G(g_1,g_2,g_3)=0$. That $\psi_G(g_1,g_2,g_3)$ is a constant function follows from the equality
$$
\psi_G(g_1,g_2,g_3) =det(M^{-1})\, \psi_F(f_1,f_2,f_3).
$$
We now assume (iii) and prove (i). Let $g$ be the genus of $X_1(N)$. Since $X_1(N)$ and $A$ are both defined over $\Q$, there exists a $\Q$-basis  $\left\{ f_1,\dots,f_g\right\}$ of  $S_2(X_1(N))=S_2(N)$ such that $\left\{ f_1,f_2,f_3\right\}$ is a $\Q$-basis of $S_2(A)$. Therefore the corresponding polynomial $F(X_1,X_2,X_3)$ has rational coefficients. Let $C$ be the smooth plane quartic $C:F(X_1,X_2,X_3)=0$. The following map
$$
\begin{array}{ccccccc}
\pi&:&X_1(N) & \rightarrow  & \mathbb{P}^{g-1} & \rightarrow &  \mathbb{P}^{2}\\
& & P & \rightarrow  & [ f_1(P):\dots :f_g(P) ]  & \rightarrow &   [f_1(P):f_2(P):f_3(P)]
\end{array}
$$
defines a non-constant $\Q$-morphism from $X_1(N)$ to $C$, i.e.  $C$ is a non-hyperellip\-tic modular curve (defined over $\Q$) of genus $3$. Now, using the same argument as above, since $\psi_F(f_1,f_2,f_3)$ is constant, it follows that $S_2(C)=S_2(A)$. In particular, this implies $\jac{C}\Qisog A$. \QED

\begin{remark}\label{obs1}{\rm  Let $J_1(N)\twoheadlongrightarrow A$ be a modular abelian $3$-fold such that $S_2(A)=\langle f_1,f_2,f_3\rangle $ and $C: F(X,Y,Z)=0$ be a smooth plane quartic such that $F(f_1,f_2,f_3)=0$ and $\psi_F (f_1,f_2,f_3)$ is not a constant function. Then the jacobian of the curve $C$ is not necessarily isogenous to the abelian variety $A$. For example: let $\{f_1,f_2,f_3\}$ be a $\Q$-basis of $S_2(A)$, where $A$ is the new modular abelian $3$-fold $A_{120A_{\{0,0,0,2\}}} \times E_{120A}$. There is a $\Q$-rational smooth plane quartic $C: F(X,Y,Z)=0$ with $F(f_1,f_2,f_3)=0$ and for which $\psi_F(f_1,f_2,f_3)$ is not a constant function. In fact, $C$ is modular and $\jac{C} \not\sim A$, but  $\jac{C}\Qisog A_{30A_{\{0,2\}}} \times X_0(15)$, i.e. $C$ is non-new of level $30$ (see Example \ref{C120}).
}\end{remark}

\begin{remark}{\rm  It can happen that for $S_2(A)= \langle f_1,f_2,f_3\rangle_\Q \subseteq  S_2 (N)$ there exists a (plane) non-hyperelliptic curve of genus $3$ with equation  $C_d: F_d(X,Y,Z)=0$ of degree $d\geq 5$ such that $F_d(f_1,f_2,f_3)= 0$. In that case,  the image $C: G(X,Y,Z)=0$ of the canonical embedding of $C_d$ determined by $\{f_1,f_2,f_3\}$ is a modular curve (since the inclusion between the function fields $ \Q (C)\subseteq \Q (X_1(N))$ implies the existence of a $\Q$-morphism from $X_1(N)$ onto $C$), however $\jac{C}$ is  not necessarily $\Qbar$-isogenous to $A$ (see Examples \ref{degree7} and  \ref{Modul243}).
}\end{remark}

\begin{lemma}\label{Lemmaalgo}
Let $C$ be a non-hyperelliptic new modular curve of genus $3$ and $\pi : X_1(N)\twoheadlongrightarrow C$ the corresponding modular parametrization. Then:
\begin{itemize}
\item[(i)] There exist $h_1,h_2,h_3\in S_2(N)$ with rational $q$-expansions 
$$
\left\{
\begin{array}{l}
h_1(q) =  q   +   O(q^2)  \\
h_2(q)  =  q^2 +  O(q^3) \\
h_3(q)  =  O(q^3)\, ,
\end{array}
\right.
$$
such that $ S_2(C)=\langle h_1,h_2,h_3 \rangle $  with
$\ord_q h_3\leq 5$. Furthermore, if $\jac{C}$ is $\Q$-simple, then $\ord_q h_3 < 5$.
\item[(ii)] Let $\phi:C\longrightarrow \mathbb{P}^2$ be the canonical embedding given by the basis of regular differentials in (i). Then $\phi(C):F(X,Y,Z)=0$ is a smooth plane quartic defined over $\Q$ with the $\Q$-rational point $P_\infty =(1:0:0)$. Moreover, $P_\infty$ is a flex (resp. a hyperflex) if  $\ord_q h_3\geq 4$ (resp. $\ord_q h_3= 5$).
\end{itemize}
\end{lemma}

\begin{remark}{\rm 
In the previous lemma, the point $P_\infty$ is never a hyperflex of $C$ if $\jac{C}$ is $\Q$-simple. When $\jac{C}$ does split over $\Q$, it may happen that $P_\infty$ is a hyperflex of $C$ (c.f. $C^A_{39A_{\{0,6\}}}$, $C^{A,B,D}_{99}$ from Table \ref{tableg3} in the Appendix).

}\end{remark}

{\bf Proof.}
(i) Our strategy is to split the proof into three cases according to the decomposition of $\jac{C}$ over $\Q$ : \\
\noindent Case \textbf{A}: $\, \jac{C}$ is $\Q$-simple. Then $\jac{C}\Qisog  A_f$ with $f\in S_2(N,1)$.\\
\noindent Case \textbf{AE}:\,  $\jac{C}\Qisog A\times E$, where $E$ is an elliptic curve over $\Q$ and $A$ is a $\Q$-simple $2$-fold. Then there exist $g\in S_2(N,1)$ such that $A_{g}\Qisog E$ and $f\in S_2(N,\varepsilon)$ such that $A_{f}\Qisog A$ and $\ord(\varepsilon)\in\{1,2,3,4,6\}$ (since $\Q(\varepsilon)\subset K_f$).\\
\noindent Case \textbf{EEE}: \,  $\jac{C}\Qisog E_1\times E_2\times E_3$, where $E_1,E_2,E_3$ are elliptic curves defined over $\Q$. Then there exist $f_i\in S_2(N,1)$ such that $A_{f_i}\Qisog E_i$, $i=1,2,3$.\\

Following \cite[Corollary 7.3 (i)]{BakGonPoo}, there exists a basis $\{g_1,g_2,g_3\}$ of $S_2(C)$ such that $g_i(q)=q+\sum_{n\geq 2}{a_n^{(i)}q^n}$, and since $C$ is non-hyperelliptic, 
\begin{equation}\label{a2}
 \mbox{\em it is not possible that $a_2^{(1)}=a_2^{(2)}=a_2^{(3)}$.}
\end{equation}
 
 We are going to apply (\ref{a2}) to our three cases: \\

\noindent Case \textbf{A}. Let $f(q)=\sum_{n\ge 1}{a_n q^n}$. Then by (\ref{a2}) we have $a_2\notin \Q$, that is,  $K_f=\Q(a_2)$. In this case $S_2(A_f)=\langle f,{}^\sigma f,{}^\beta f  \rangle$, where $\{id,\sigma,\beta\}$ are the $\Q$-embeddings of $K_f$ into $\Qbar$. Now we construct an explicit $\Q$-basis $\{h_1,h_2,h_3\}$ for $S_2(A_f)$. Let $p(x)=x^3+ax^2+bx+c$ be the minimal polynomial of $a_2$ and 
$$
g_i(q)=\frac{1}{3}\,\sum_{n\ge 1} \mbox{Tr}_{K_f/\Q}(a_2^{i-1}\,a_n)q^n\, ,\qquad i=1,2,3.
$$
 Then $g_1,g_2,g_3$ have rational $q$-expansion and $S_2(C)=\langle g_1,g_2,g_3\rangle$. Now, let be
$$
\begin{array}{l}
\displaystyle h_1(q)\!=\!g_1(q)\, ,\\
\displaystyle h_2(q)\!=\!\frac{18}{\mbox{disc}(p'(x))}(g_2(q)+\frac{a}{3}g_1(q))\, ,\\
\displaystyle h_3(q)\!=\!g_3(q)\!-\!\frac{1}{3}(a^2\!-2b)g_1(q)\!+\!\frac{1}{9}(2a^3\!-7ab\!+9c)g_2(q)\!=\!A\gamma_3 q^3\!+A q^4\!+ O(q^5),\\
\end{array}
$$
where $a_3=\alpha_3+\beta_3\,a_2+\gamma_3\,a_2^2$, since $a_3\in K_f=\Q(a_2)$, and
$$
A=\frac{2}{3}\frac{\mbox{disc}(p(x))}{\mbox{disc}(p'(x))}.
$$
Since $p(x)$ has three different real roots, both $\mbox{disc}(p(x))\neq 0$ and $\mbox{disc}(p'(x))\neq 0$, and therefore $A\in \Q^\ast$. After normalizing, we have $h_3(q)=q^3+O(q^4)$ when $\gamma_3\ne 0$ and $h_3(q)=q^4+O(q^5)$ when  $\gamma_3= 0$.

\noindent Case \textbf{AE}. Let be $K_f=\Q(\sqrt{d})$, $\aut{K_f}=\{id,\sigma\}$ and 
$$
\begin{array}{rcl}
f(q)&=&q+(A_2+B_2\sqrt{d})q^2+O(q^3)\, ,\\
{}^\sigma f(q)&=&q+(A_2-B_2\sqrt{d})q^2+O(q^3)\, ,\\
g(q)&=&q+c_2q^2+O(q^3)\, ,
\end{array}
$$
with $A_2, B_2, c_2 \in \Q\,.$

Then by (\ref{a2}) there are two possibilities depending on whether $B_2=0$ or $B_2\ne 0$.
\begin{itemize}
\item If $B_2\ne 0$, then we choose
$$
\begin{array}{l}
h_1(q)=\frac{f(q)+{}^\sigma f(q)}{2}=q+O(q^2)\, ,\\
h_2(q)=\frac{f(q)-{}^\sigma f(q)}{2 B_2\sqrt{d}}=q^2+O(q^3)\, ,\\
h_3(q)=g(q)-h_1(q)-(c_2-A_2)h_2(q) =O(q^3)\, .
\end{array}
$$

\item If $B_2=0$, from (\ref{a2}) we have $A_2\ne c_2$, hence
$$
\begin{array}{l}
h_1(q)=\frac{f(q)+{}^\sigma f(q)}{2}=q+O(q^2)\, ,\\
h_2(q)=\frac{g(q)-h_1(q)}{c_2-A_2}=q^2+O(q^3)\, ,\\
h_3(q)=\frac{f(q)-{}^\sigma f(q)}{2 \sqrt{d}}=O(q^3)\, .\\
\end{array}
$$
\end{itemize}

\noindent Case \textbf{EEE}. Let
$$
\begin{array}{l}
f_1(q)=q+a_2q^2+O(q^3)\, ,\\
f_2(q)=q+b_2q^2+O(q^3)\, ,\\
f_3(q)=q+c_2q^2+O(q^3)\, ,
\end{array}
$$
with $a_2, b_2, c_2 \in \Q\,.$

Then by (\ref{a2}) we can assume that $a_2\ne b_2$, and thus
$$
\begin{array}{l}
h_1(q)=f_1(q)=q+O(q^2)\, ,\\
h_2(q)=\frac{f_1(q)-f_2(q)}{a_2-b_2}=q^2+O(q^3)\, ,\\
h_3(q)=f_3(q)-f_1(q)-(c_2-a_2)h_2(q) =O(q^3)\, .
\end{array}
$$
Let $n\!=\!\ord_{\!q} h_3$. For all the cases above, the degree four monomial $h_1^ih_2^jh_3^{4-i-j}$ has order $1+2j+n(4-i-j)$. 
It is easy to check that for $n\geq 6$ all these orders are different, and hence there is no $F\in \Q[X,Y,Z]$ of degree $4$ such that $F(h_1,h_2,h_3)=0$.

(ii) Let $S_2(C)$ be choosen as above, then the image of $C$ by the canonical embedding $\phi$ is a smooth plane quartic of the form
\begin{equation} 
\sum_{i,j,k\in\Z_{\ge 0}}^{i+j+k=4} a_{ijk} X^iY^jZ^k =0\, ,\qquad \mbox{where $a_{ijk} \in \Q.$} \nonumber
\end{equation}
\noindent  For $X:=\!h_1(q),Y:=\!h_2(q)$ and $Z:=\!h_3(q),$ the degree four monomials $X^iY^jZ^k$ have $q$-expansions $ X^4=q^4+O(q^5), {}~X^3Y=q^5+O(q^6),$ and $X^iY^jZ^k =O(q^6)$ otherwise. It follows that $a_{400},a_{310}=0,$ and therefore $P_\infty:=(1:0:0) \in \phi (C)(\Q). $ The tangent line $l_\infty $ at $P_\infty$ is the line with equation $Z=0.$ Furthermore, for $h_3(q)= q^4 + O(q^5)$ we have $ X^2Y^2=q^6+O(q^7)$ and $ X^iY^jZ^k =O(q^7)$ for all the degree four monomials different from $ X^2Y^2,X^4,X^3Y.$ Hence $ a_{220}=0$. In this case, $P_\infty$ is at least an ordinary flex. Similarly, for $h_3(q)=q^5+O(q^6)$ the degree four monomials have $q$-expansions $X^2Y^2=q^6+O(q^7), {}~XY^3=q^7+O(q^8)$ and $X^iY^jZ^k =O(q^8)$ for degree four monomials different from $X^2Y^2,X^4,X^3Y,XY^3.$ In this case,  $a_{220}=a_{130}=0$, and therefore $P_\infty$ is an hyperflex, which proves the assertion.
\QED
\subsection{Computational Algorithm}
Proposition \ref{nonhyper3} provides us with a theoretical algorithm to recognize whether or not a modular abelian $3$-fold corresponds to a non-hyperelliptic modular curve of genus $3$. The following result gives us a computational algorithm\,:

\begin{proposition}\label{compualgo}
Let $J_1(N)\twoheadlongrightarrow A$ be a modular abelian $3$-fold  defined over $\Q$ and let $S_2(A)=\langle h_1(q),h_2(q),h_3(q)\rangle_{\Q}$. If there exist a non-singular homogeneous polynomial $F\in \Q[X,Y,Z]$ of degree $4$ and a constant $c_F\in \Q^\ast $ such that 
\begin{itemize}
\item[(i)] $F(h_1(q),h_2(q),h_3(q))=O(q^{c_N})$, where $\displaystyle c_N=\frac{2}{3}[SL_2(\Z)\,:\,\Gamma_1(N)] $ ,
\item[(ii)] $\psi_F(h_1(q),h_2(q),h_3(q))= c_F +O(q^{c'_N}) $,  where $\displaystyle c'_N=\frac{1}{2}[SL_2(\Z)\,:\,\Gamma_1(N)]  $ ,
\end{itemize}
then the curve $C:F(X,Y,Z)=0$ is a non-hyperelliptic modular curve of level $N$ such that $\jac{C}\Qisog A$. 
\end{proposition}
 
\begin{remark}{\rm 
Let $A$ be a new modular abelian variety of level $N$. If $A$ is a quotient of $J_0(N)$  then $\Gamma_1(N)$ could be replaced by $\Gamma_0(N)$ in the formulas for $c_N$ and $c'_N$. The only case when it could not be replaced is when $A\Qisog A_f\times A_g$ and $f\in S_2(N,\varepsilon)$ with $\varepsilon$ non-trivial. In that case\footnote{See \cite[\S 6]{GonGon} for the definition of the congruence subgroup $\Gamma(N,\varepsilon)$  and the corresponding modular curve $X(N,\varepsilon)$.},  $A$ is a quotient of $\jac{X(N,\varepsilon)}$ and $\Gamma_1(N)$ could be replaced by $\Gamma(N,\varepsilon)$ in the formulas for $c_N$ and $c'_N.$
}\end{remark}

{\bf Proof.}
For a positive integer $k$, it is well known that if $f\in S_k(N)$ and $f(q)=O(q^c)$ with $c\ge\frac{k}{12}[SL_2(\Z)\,:\,\Gamma_1(N)]$ then $f=0$. Using condition (i), we apply this result to the modular form $F(h_1,h_2,h_3)\in S_8(N)$ to prove that $F(h_1(q),h_2(q),h_3(q))=0$. Now, using (ii) we are going to prove that  $\psi_F(h_1,h_2,h_3)= c_F$. In order to prove it, we first observe that $h'_1h_3 - h_1h'_3\in S_6(N)$, since $h_1, h_3\in S_2(N)$. On the other hand, the modular form $\frac{\partial F}{\partial Y}(h_1,h_2,h_3)$ belongs to $S_6(N)$ since $\frac{\partial F}{\partial Y}(X,Y,Z)$ is a homogeneous polynomial of degree $3$ . Then $G=h'_1h_3 - h_1h'_3- c_F\frac{\partial F}{\partial Y}(h_1,h_2,h_3)\in S_6(N)$. Therefore,  $\psi_F(h_1,h_2,h_3)= c_F$ if and only if $G(q)=O(q^{c'_N})$, that is, if and only if (ii) holds. Then the proof of the proposition follows directly from Proposition \ref{nonhyper3}.
\QED

To compute a model for the modular curve $C$ defined over the integers, let $\left\{h_1,h_2,h_3 \right\}$ be a basis of $S_2(C)$ consisting of cusp forms with integral $q$-expansions. Let us consider an enumeration
$$
\{f_1,\dots, f_{15} \}=\{h_1^{i} h_2^{j}h_3^{k} \in S_8(N)\,|\,i,j,k\in\Z_{\ge 0},\, i+j+k=4 \}
$$ of the set of degree four monomials in $h_1,h_2,h_3,$ and let $f_i(q)=\sum_{j\geq 1} b_{ij} q^j$ be the $q$-expansions of the $f_i$.  The smooth plane quartic $F$ defining the modular curve $C$, i.e. with $F(h_1,h_2,h_3)=0,$ is given by the cusp form $F=\sum_{i=1}^{15} a_i f_i = \sum_{j\geq 1}\left( \sum_{i=1}^{15} b_{ij} a_i\right) q^j \in S_8(N),$ whose defining coefficients $a_i$ are computed by solving the linear equation $B^{T} \cdot a=0, $ where $a=(a_i)$ is a non-trivial vector with $15$ entries, $B=(b_{ij})$ a matrix with $15$ rows and at least $c_N$ columns.

In the following example we are going to explain how this method works in practice.
\begin{example}{\rm 
Let $A$ be the new modular abelian $3$-fold $A_{243E}$. The vector space $S_2(A)$ is generated by
the $q$-integral basis $\left\{h_1,h_2,h_3 \right\},$ where
\begin{small}
\begin{eqnarray*}
h_1(q) &=& q - 3 q^5 - 2 q^7 - 3 q^8 - 2 q^{10} + q^{13} + q^{16} - 3 q^{17} + 2 q^{19} + O(q^{20})\, , \\
h_2(q) &=& q^2 - q^5 - 3 q^7 - 4 q^8 - 3 q^{10} + 4 q^{11} + 3 q^{13} + 2 q^{14} + 3 q^{16} + 6 q^{19} + O(q^{20})\, , \\
h_3(q) &=&  q^4 - 2 q^7 - 3 q^8 - q^{10} + 3 q^{11} + q^{13} + 3 q^{14} + 3 q^{16} + 3 q^{19} + O(q^{20}) \, .
\end{eqnarray*}
\end{small}\noindent By Proposition \ref{compualgo} we compute the $q$-expansions of $h_1,h_2,h_3$ with $c_{203}=180$ coefficients and then an equation of the modular curve $C_{243}^E$ is
\begin{small}
$$
C_{243}^E \,:\,    X^3 Z - 3 X^2 Z^2 - X Y^3 + 9 X Y Z^2 - 6 X Z^3       + 2 Y^3 Z - 9 Y^2 Z^2 + 9 Y Z^3 - 2 Z^4         =0 \, .
$$
\end{small}
\noindent Even if we don't need  all the $180$ coefficients to solve the linear system, we will need them for verifying $F(h_1,h_2,h_3)=0$.  Since $\psi_F(h_1,h_2,h_3)=1$, the jacobian $\jac{C_{243}^E}$ is $\Q$-isogenous to $A$. \QED
}\end{example}

\subsection{Automorphisms of new modular curves}\label{auto_new}

Let $C$ be a new modular non-hyperelliptic curve of genus $3$ and level $N$. The proof of Lemma \ref{Lemmaalgo} provides a full description of the splitting behaviour (over $\Q$) of $\jac{C}$. In the case that $C$ could not be parametrized by $X_0(N)$, we have that  $\jac{C}\Qisog A_g\times A_f$ where $f$ and $g$ are newforms of level $N$ for which the Nebentypus of $f$ is trivial and the Nebentypus of $g$ is of order $1,2,3,4$ or $6$. In this section we prove that the Nebentypus of $g$ can never be $6$. This result is important since it reduces the complexity for the computation of all new modular non-hyperelliptic curves of genus $3$ and for a fixed level $N$.

Let $C$ be a new modular curve of level $N$ and genus $g\geq 2$. The diamond operators $\langle d \rangle$  on $X_1(N)$ induce automorphisms of $C$ over $\Qbar$. Let $\calD$ be the abelian subgroup of ${\rm Aut}_{\Qbar}(C)$ consisting of diamond automorphisms, then $\calD$ is $\gal{\Qbar/\Q}$-stable. Moreover, there exists a surjective morphism $X_0(N)\twoheadlongrightarrow C$ if and only if $\calD$ is the trivial group.  

Let $\calD'$ be a  subgroup of $\calD$. If the curve $C' = C / \calD'$ has genus $g'$, then $g-g'$ is even (see \cite[Lemma 6.17]{BakGonPoo}). In particular, if $C$ is a new modular curve of genus $3$ and $\calD' \neq \{1\}$, then $g'=1$. Therefore, if $C$ is non-hyperelliptic, Proposition \ref{quotg1} gives us the following result:

\begin{lemma}
Let $C$ be a new modular non-hyperelliptic curve of genus $3$. Then $\calD$ is either trivial or cyclic of order $2,3$ or $4$.
\end{lemma}

\begin{corollary}
Let $C$ be a new modular non-hyperelliptic genus $3$ curve of level $N$ such that $\jac{C}$ is not a quotient of $J_0(N)$. Then $\jac{C}\Qisog A_f \times A_g $, where $f$ is a newform of level $N$ with trivial Nebentypus and $g$ is a newform of level $N$ with Nebentypus of order $2,3$ or $4$. 
\end{corollary}
\section{Non-new non-hyperelliptic modular curves of genus $3$}\label{section_non_new}

If  $C$ is a non-new non-hyperelliptic modular curve of genus $3$, Proposition \ref{nonhyper3} could still be used to compute a rational equation of $C$. The following examples present some behaviour that could appear when dealing with non-new non-hyperelliptic modular curves of genus $3$.

\begin{example}{\rm \label{nichtneu}
Let $f$ and $g$ be the newforms attached to the modular abelian varieties  $A_{178C}$ and  $E_{89A}$ respectively. Their $q$-expansions begin as follow:
$$
\begin{array}{l}
f(q)= q - q^2 + aq^3 + q^4 + (-2a - 3)q^5  - a q^6 - 2 q^7 - q^8 + O(q^{9}),\\
g(q)	=  q - q^2 - q^3 - q^4 - q^5 + q^6 - 4 q^7 + 3 q^8 - 2 q^9 + q^{10} + O(q^{11})  ,
\end{array}
$$
where $K_{f}=\Q(a) $, $a^2 + 2a - 1 =0$, $\aut{K_f}=\{{\rm id},\sigma\}$ and $K_{g}=\Q$. Let $\{ f_1,f_2\}$ be the $\Q$-basis of $S_2(A_{178C})$ with the following $q$-expansion: 
$$
\begin{array}{l}
f_1(q)= q - q^2 + q^4 - 3 q^5 - 2 q^7 - q^8 - 2 q^9 + 3 q^{10} - 2 q^{13} + 2 q^{14} + O(q^{15})\,,\\
f_2(q)=    q^3 - 2 q^5 - q^6 - 2 q^9 + 2 q^{10} + 2 q^{11} + q^{12} + q^{15} + 2 q^{17} + 2 q^{18} +  O(q^{19})\,.\\
\end{array}
$$
Let be $f_3(q)=g(q)+2g(q^2)$. Then the non-singular irreducible homogeneous polynomial 
$$ 
F(X,Y,Z)\!=\!X^3 Z -X^2Y^2 + 2XY^2Z - Y^3Z - 4X^2Z^2 + 3XYZ^2 -  2Y^2 Z^2 + 3 X Z^3 - Y Z^3
$$
satisfies $F(f_{1},f_{2},f_3)=0$ and $\psi_F(f_{1},f_{2},f_3)=1$. Therefore the smooth plane quartic $C^{89A}_{178C}:F(X,Y,Z)=0$ is a non-hyperelliptic and non-new modular curve of genus $3$ such that $S_2(C^{89A}_{178C})=\langle f_1,f_2,f_3\rangle $. In particular, $\jac{C^{89A}_{178C}}$ is $\Q$-isogenous to $A_{178C}\times E_{89A}$.
\QED}\end{example}

The following couple of examples show  (plane) non-hyperelliptic curves of genus $3$ with equations  $C_d: F_d(X,Y,Z)=0$ of degree $d\geq 5$.

\begin{example}{\rm \label{degree7}
Let $A$ be the new modular abelian $3$-fold $A_{178D}$.  With respect to the integral basis $\{f_{1},f_{2},f_{3}\}$ of $S_2 (A)$ 
$$
\begin{array}{l}
f_{1}(q)\! =\!    q + q^2 + q^4 + q^8 + 3q^9\!+ 2q^{11}\!- 6q^{15}\!+ q^{16}\!- 2q^{17}\!+ 3q^{18}\! - 4q^{19}\!+ O(q^{20})\,,\\
f_{2}(q) \!=\!   q^3 - q^5 + q^6 - q^9 - q^{10} + q^{12} - 2 q^{13} + q^{15} + q^{17}\!- q^{18}\!+ q^{19}\!+ O(q^{20})\,,\\
f_{3}(q)\!= \!  q^7 - 2 q^9 - q^{13} + q^{14} + 2 q^{15} + 2 q^{17} - 2 q^{18} +\!O(q^{20})\,,
\end{array}
$$
we compute the equation of a genus $3$ non-hyperelliptic curve $C :  F_7(X,Y,Z)\!=\!0$, where
$$
\begin{array}{l}
F_7(X,\!Y,\!Z)\!=\!X^5\!Z^2\!- 3X^4Y\!Z^2\! + 8X^4\!Z^3\! - \!2X^3Y^3\!Z\! +\!7X^3Y^2\!Z^2\! - 23X^3Y\!Z^3\!- \!Y^7\!      \\
    \qquad \qquad \quad\!+26X^3Z^4\! - 3X^2Y^3Z^2\! + 18X^2Y^2Z^3\! - 53X^2YZ^4\! + 42X^2Z^5\! + XY^6  \\
    \qquad  \qquad\quad \!- 3XY^5Z\! +\!XY^4Z^2\!+\!14XY^3Z^3\!- \!10XY^2Z^4\!- \!36XYZ^5\!+ 32XZ^6  \\
    \qquad  \qquad   \quad\!+ 4Y^6Z\!+ 10Y^5Z^2\!-\!66Y^4Z^3\!+\!124Y^3Z^4\!-\!100Y^2Z^5\!+ 20YZ^6\!+ 8Z^7 ,\\
\end{array}
$$
for which $F_7(f_1,f_2,f_3)=0$. A canonical embedding $\phi$ of $C$ is computed using {\sc Magma}, and $C'=\phi (C)$ has a quartic model with equation
$$
\begin{array}{l}
C'\,\,: \,\,X^4 - 4X^3Y + 6X^3Z + 2X^2Y^2 - 5X^2Z^2 + 4XY^3 - 30XY^2Z  \\
    \qquad \quad+ 64XYZ^2 - 42XZ^3 - 3Y^4 + 8Y^3Z + 9Y^2Z^2 - 40YZ^3 + 28Z^4=0\,. 
\end{array}
$$
We have been able to check $\jac{C'}\stackrel{\mathbb{F}_p}{\sim}\jac{C^{89A}_{178C}}$ from Example \ref{nichtneu} for $p\nmid 178$ such that $p<500$. Furthermore, by computing the absolute Dixmier-Ohno invariants (cf. \cite{Dixmier,GirardKohel}), we have $C'$ is $\Qbar$-isomorphic to $C^{89A}_{178C}$. 

The abelian variety $A_{178D}$ is $\Qbar$-simple (since $f_{178D}$ has no extra-twists and has no complex multiplication), therefore it is not $\Qbar$-isogenous to $\jac{C'}$. 
\QED}\end{example}

If  $A$ is a new modular abelian variety $A_f$ and we build a smooth plane quartic $C'$ as above, then $\jac{C'}$ and $A_f$ will in general not be $\Qbar$-isogenous. However, if $\jac{C'}$ is  $\Q$-simple and new of level $N$, then there exists a newform $g$ of level $N$ with $\jac{C'}\Qisog  A_g$. Moreover, if $g = f \otimes \chi $ for some Dirichlet-character $\chi$ then $A_f \stackrel{K}{\sim} \jac{C'}$, where $K=\Qbar^{{\rm ker}\,\chi}$ (cf. \cite{Shimura73}).

\begin{example}{\rm \label{Modul243}
Let $A$ be the new modular abelian $3$-fold $A_{243F}$. Then there exists an integral basis $\{f_1 ,f_2, f_3\} $ of $S_2(A)$ satisfying $F_6(f_1,f_2,f_3)=0,$ where 
$$
\begin{array}{l}
F_6(X,Y,Z) \!=\! X^5Z - 7X^4Z^2 - X^3Y^3 - 9X^3YZ^2 + X^3Z^3 + 6X^2Y^3Z\\
   \qquad \qquad \qquad  + 19X^2Z^4 - 3XY^3Z^2 + 18XY^2Z^3 + 27XYZ^4 + 2XZ^5  \\
   \qquad \qquad \qquad   + 27X^2YZ^3 + 9X^2Y^2Z^2 + 8Y^3Z^3 + 9Y^2Z^4 - 9YZ^5 - 8Z^6\, ,
\end{array}
$$
 defines a non-hyperelliptic plane curve $C\, : \, F_6(X,Y,Z)=0$ of genus $3$ and degree $d=6$.  A canonical embedding of $C$ is a smooth plane quartic $C'$ given by
$$
\begin{array}{l}
C'\,\,: \,\,  X^3Y  - 12X^2Y^2 + 9X^2YZ - 24X^2Z^2 + 48XY^3 + 24XY^2Z -57XYZ^2\\
  \qquad \quad  - 2X^3Z + 66XZ^3 - 64Y^4 + 104Y^3Z - 36Y^2Z^2 - 65YZ^3 + 88Z^4 =0\, .
\end{array}
$$
 We have $f_{243E}=f_{243F}\otimes\chi$ where $\chi$ is the non-trivial Dirichlet-character of level $3$, and thus $A_{243E}\stackrel{\Q(\sqrt{-3})}{\sim}A_{243F}$. In fact, the quartic $C'$ is modular and its jacobian ${\rm Jac}(C')$ is conjecturely $\Q$-isogenous to $ A_{243E}\stackrel{\Q}{\sim}\jac{C_{243}^E}$ (we have checked   $\jac{C'}\stackrel{\mathbb{F}_p}{\sim} A_{243E}$ for primes $p$ with $3<p<500$).\QED
}\end{example}


The next example illustrates in particular why condition (iii)(b) in Proposition \ref{nonhyper3} is necessary.
\begin{example}{\rm \label{C120}
Let $A_1$ be the new modular abelian variety of level $120$ that is the product of $A_{{120A}_{\{ 0, 0, 0, 2 \}}}$ and $E_{120A}$  and let $C_1$ be the smooth plane quartic defined by $F(X,Y,Z)=0$, where 
$$
F(X,Y,Z)\!=\!3\,X^4 - 10\,X^2\,Y^2 - 10\,X^2\,Z^2 + 7\,Y^4 + 8\,Y^3\,Z + 2\,Y^2\,Z^2 - 8\,Y\,Z^3 + 7\,Z^4.
$$
There exists a basis $\{f_1,f_2,f_3\}$ of $S_2(A_1)$ such that $F(f_1,f_2,f_3)=0$ and $\psi_F(f_1,f_2,f_3)$ is non-constant.

Similarly, consider the new modular abelian three fold $A_2=A_{{240B}_{\{ 0, 0, 0, 2 \}}} \times E_{240A}$ and the smooth plane quartic $C_2:G(X,Y,Z)=0$, where
$$
G(X,Y,Z)\!=\!3\,X^4 - 10\,X^2\,Y^2 - 10\,X^2\,Z^2 + 7\,Y^4 - 8\,Y^3\,Z + 2\,Y^2\,Z^2 + 8\,Y\,Z^3 + 7\,Z^4.
$$
There exists a basis $\{g_1,g_2,g_3\}$ of $S_2(A_2)$ such that $G(g_1,g_2,g_3)=0$ and $\psi_G(g_1,g_2,g_3)$ is non-constant.

Furthermore, $F(X,-Y,Z)=G(X,Y,Z)$, and hence $C_1\stackrel{\Q}{\simeq}C_2$.  However  $A_1$ is not  $\Q$-isogenous to $A_2$. In fact, if we denote by $\chi_1$ (resp. $\chi_2$) the Dirichlet character attached to the quadratic  field $\Q(\sqrt{-5})$ (resp. $\Q(\sqrt{-1})$), the following equalities hold
$$
f_{{120A}_{\{ 0, 0, 0, 2 \}}}= f_{{240B}_{\{ 0, 0, 0, 2 \}}} \otimes \chi_1 \quad\mbox{and}\quad f_{120A}= f_{240A}\otimes \chi_2.
$$
Therefore
$$
A_{{120A}_{\{ 0, 0, 0, 2 \}}} \stackrel{\Q(\sqrt{-5})}{\sim} A_{{240B}_{\{ 0, 0, 0, 2 \}}}\quad\mbox{and}\quad E_{120A} \stackrel{\Q(\sqrt{-1})}{\sim} E_{240A}.
$$
Furthermore, $\jac{C_i}$ is not $\Qbar$-isogenous to the modular $3$-fold $A_i$ for $i=1,2$. Nevertheless, $C_1$ (and $C_2$) must be modular for some level $M$ dividing $120$. Indeed, $C_1$ is $\Q$-isomorphic to a non-new modular curve $C$ of level $30$ with $\jac{C}\Qisog A_{{30A}_{\{ 0, 2 \}}} \times E_{15A}$. More precisely, there exists a $\Q$-basis $\{h_1,h_2\}$ for  $S_2(A_{{30A}_{\{ 0, 2 \}}})$ as well as a $\Q$-basis $\{h_3\}$ for $S_2(E_{15A})$ and a homogeneous polynomial  
$$
H(X,Y,Z)= 7 X^4 + 8 X^3 Y + 2 X^2 Y^2 - 10 X^2 Z^2 - 8 X Y^3 + 7 Y^4 - 10 Y^2 Z^2 + 3 Z^4
$$
such that  if we denote by $r_3(q)=h_3(q)+2 h_3(q^2)$ we have  $H(h_1,h_2,r_3) = 0$ and $\psi_{H} (h_1,h_2,r_3)=1$, in particular $S_2(C)=\langle h_1,h_2,r_3\rangle$. In fact, $F(Z,Y,-X)=H(X,Y,Z)$. \QED
}\end{example}


\section{Conclusion}

In this work, we present a method to compute equations for modular non-hyperelliptic curves of genus $3$. In particular, given a modular abelian $3$-fold $A$ of level $N$, we provide a criterion that enables us to check (from a basis of $S_2(A)$) whether there is a modular non-hyperelliptic curve $C$ of the same level for which $\jac{C} \Qisog A.$ We apply this method to compute a  list of $44$ new modular non-hyperelliptic curves of genus $3$ and level $N$ smaller than a certain large bound (see Appendix).  In view of the computed examples and the result (and guess) in the hyperelliptic case  \cite{GonGon, BakGonPoo}, we think it is likely that the computed list is complete. However, we are currently unable to prove this guess. It seems also impossible to prove it nowadays using the techniques proposed in \cite{GonGon} for the genus two case.


We also stress the following open question:

\emph{For every fixed genus $g\geq 2$, is the number of modular curves (without the restriction  new) of genus $g$ infinite?}

The authors of \cite{BakGonPoo} conjectured that the answer is {\it yes}.

As first step to answer the question above, Section \ref{section_non_new} explained by means of some specific example the difference that may appear between new and non-new modular curves of genus $3$.\\[1cm]

{\bf Acknowledgements:} The first author thanks the Centre of Applied Cryptographic  Research \!(CACR) of the University of Waterloo  for the facilities during his stay. The second author thanks the Universidad Aut{\'o}noma de Madrid for their hospitality during the time when this work was initiated. Further, he thanks the Centre of Applied Cryptographic  Research (CACR) of the University of Waterloo for providing a stimulating research environment. Both authors would like to thank Josep Gonz\'alez, Gabino Gonz\'alez-Diez, Adolfo Quir\'os and Christophe Ritzenthaler  for very useful comments. We thank also Francesc Bars for suggesting references. Finally, we thanks the anonymous referee for useful suggestions.

\section{Appendix}

\subsection*{Labelling} 

We remind the notation introduced in \cite[Appendix]{BakGonPoo} to label newforms of weight $2$, and in particular to fix an ordering of (Galois conjugacy classes of) newforms having a given level and Nebentypus. 

Let $f$ be a newform of level $N$ and Nebentypus $\varepsilon$, we associate to $f$ a label of the form $NX_\varepsilon$, where $X$ is a letter or string in $\{A,B,\dots,Z,AA,BB,\dots\}$. If $\varepsilon=1$, we omit the subscript $\varepsilon$ and use a label of the form $NX$, and if the Fourier coefficients of $f$ are integers
we will use the labeling in~\cite{cremona97}. First we remember how is $X$ constructed from $f$; Fix $N$ and $\varepsilon$. Let $f = \sum a_n q^n \in  S_2(N,\varepsilon)$ be a newform. To $f$ associate the infinite sequence of integers ${\mathbf t}_f = ( \mbox{Tr}_{K_f/\Q} a_1, \mbox{Tr}_{K_f/\Q} a_2, \dots)$. Choose $X \in \{A,B,\dots,Z,AA,BB,\dots\}$ according to the position of ${\mathbf t}_f$ in the set $\{\, {\mathbf t}_g : g \in S_2(N,\varepsilon)\,\mbox{newform} \,\}$ sorted in increasing dictionary order. Notice that ${\mathbf t}_f$ determines the Galois conjugacy class of $f$. 

Now, we describe how the Dirichlet character $\varepsilon\colon (\Z/N\Z)^*\rightarrow \C^*$ is encoded. Let $N=\prod{p_n^{\alpha_n}}$ be the prime-ordered factorization. Then there exist unique $\varepsilon_{p_n}\colon (\Z/p_n^{\alpha_n}\,\Z)^*\rightarrow \C^*$ such that $\varepsilon=\prod{\varepsilon_{p_n}}$. If $p$ is an odd prime, let $g_p$ be the smallest positive integer that generates $(\Z/p^{\alpha}\,\Z)^*$, and if $p=2$ and $\alpha \le 2$, let $g_p=-1$; in these cases $\varepsilon_{p}$ is determined by the integer $e_p \in [0,\varphi(p^\alpha))$ such that $\varepsilon_p(g_p)=e^{2\pi i e_p/\varphi(p^\alpha)}$.
If $p=2$ and $\alpha > 2$, then $\varepsilon_2$ is determined by the integers $e_2',e_2'' \in [0,\varphi(2^\alpha))$ such that $\varepsilon_2(-1)=e^{2\pi i e_2'/\varphi(2^\alpha)}$ and $\varepsilon_2(5)=e^{2\pi i e_2''/\varphi(2^\alpha)}$,
and we write $e_2 = \{e_2',e_2''\}$.  Assuming that $N$ is implicit, we denote $\varepsilon$ by $\{e_p\,:\,p|N\}$.

If $f\in S_2(N,\varepsilon)$ is a newform with label $NX_{\varepsilon}$, then $A_{NX_{\varepsilon}}$ will denote the corresponding modular abelian variety $A_f$, except that when $\dim A_f=1$, we instead follow the labeling in~\cite{cremonaweb} and use the letter $E$ instead of $A$ to denote the modular elliptic curve $A_f$.

\subsection*{Tables} 
Performing the computations of all the non-hyperelliptic new modular curves of genus $3$ would be extremely time-consuming. Therefore we have conducted a search of all non-hyperelliptic new modular curves of genus $3$ and some fixed level. For the case {\bf A} up to level $N\leq 10000$, for the case {\bf AE} up to level $N\leq 4000$ and for the case {\bf EEE} up to level $N\leq 130000$, the highest level in Cremona's tables \cite{cremonaweb} (see proof of Lemma \ref{Lemmaalgo} for the notation of cases {\bf A},   {\bf AE} and {\bf EEE}). For this aim, we have implemented the method developed at Section \ref{section3} in {\sc Magma} \cite{Magma} using W.~A.~Stein's Modular packages. We have obtained a total of $44$ such curves that appear at Table \ref{tableg3}.

Table \ref{tableg3} has two columns. The first one shows the label of the non-hyperelliptic new modular curve of genus $3$ and the second column shows the correponding smooth plane quartic model over $\Z$. The notation for the curves is as follows:
 
{\bf Notation. }{\rm  Let $C$ be a new modular non-hyperelliptic genus $3$ curve of level $N$. In the case that $\jac{C}$ is a factor of $J_0(N)$ we will add $N$ as a subscript and the corresponding letters of the labels corresponding to the $\Q$-factors of   $\jac{C}$ as superscripts to $C$. Otherwise,  $\jac{C}\Qisog  A_f\times E$, where $f\in S_2(N,\varepsilon)$ such that $\varepsilon$ is not trivial and $E$ is an elliptic curve over $\Q$. We will denote this modular curve by $C^{X_E}_{{N\,X_A}_{\varepsilon}}$ where ${X_{A}}_\varepsilon$ and $X_E$ are the corresponding labels for $A_f$ and $E$ respectively.\\
}

Finally, we have added the superscript $\blacklozenge$  (resp.  $\bigstar$) on the left of the labelling of the curve if  $P_\infty$ is an ordinary flex (resp.  hyperflex).

Note that we have not attempted to reduce the size of the coefficients appearing in the computed models. However the models obtained have already very small coefficients: the worst-case is the modular curve $C_{65}^{A,B},$ which has largest coefficient $98$. 

Contrary to what is observed in \cite{BakGonPoo} for hyperelliptic new modular curves of genus $3,$ there exists a non-hyperelliptic new modular curve of genus $3$ and level $N$ for which $N$ has more than two different odd prime divisors, namely the modular curve $C_{855}^H$. In fact, this is the only curve for which three different primes appear in the factorization of the level of modularity. In all other cases, there are just one or two primes.

As a final remark, note that the curve $C^A_{{49,A}_{\{ 14 \}}}$ is the Klein quartic, it is thus $\Q$-isomorphic to the classical modular curve $X(7)$.
{\small
\setlongtables
\begin{longtable}{Z}
\caption[New modular non-hyperelliptic genus $3$ curves]{New modular non-hyperelliptic genus $3$ curves}\label{tableg3}\\
\toprule
 $C$ & $\,\,\,\,\,\,F(x,y,z)\,\,=\,0 $\\
\midrule
\endfirsthead
\toprule
 $C$ & $\,\,\,\,\,\,F(x,y,z)\,\,=\,0 $\\
\midrule
\endhead
\midrule
\endfoot
\bottomrule   
\endlastfoot
$C^A_{{20,A}_{\{ 1, 1 \}}}$ & $x^3z - x^2y^2 - 3x^2z^2 + xy^3 + 4xz^3 - 2z^4=0$ \\
$C^A_{{24,A}_{\{ 0, 1, 0 \}}}$ & $x^3 z - x^2 y^2 - x^2 z^2 + x y^3 - x y^2 z - 3 x y z^2 + y^3 z + 2 y^2 z^2 + y z^3=0$ \\
${}^{\blacklozenge}C^A_{{24,A}_{\{ 1, 1, 1 \}}}$ & $x^3 z - 2 x^2 y z - x^2 z^2 - x y^3 + 2 x y^2 z + 6 x y z^2 + 2 y^3 z - 
2 y^2 z^2 - 4 y z^3=0$ \\
${}^{\blacklozenge}C^A_{{36,A}_{\{ 1, 3 \}}}$ & $x^3 z - 3 x^2 z^2 - x y^3 + 4 x z^3 + 2 y^3 z - 2 z^4=0$ \\
${}^{\bigstar}C^A_{{39,A}_{\{ 0, 6 \}}}$ & $x^3 z - 2 x^2 z^2 + 4 x y^2 z - 7 x y z^2 - 6 x z^3 - y^4 + 5 y^3 z + 2 y^2 z^2 - 6 y z^3 - 3 z^4=0$
\\
${}^{\blacklozenge}C^A_{{39,A}_{\{ 0, 4 \}}}$ & $x^3 z - 2 x^2 y z - x y^3 - 2 x y^2 z + 2 x y z^2 + y z^3=0$ \\
$C^{A,B}_{43}$ & $2 x^3 z - 2 x^2 y^2 - 6 x^2 z^2 + x y^3 + 9 x y^2 z - 5 x y z^2 + 11 x z^3 - 9 y^4 + 12 y^3 z - 22 y^2 z^2 + 12 y z^3 -9 z^4=0$\\
${}^{\blacklozenge}C^A_{{45,A}_{\{ 2, 0 \}}}$ & $x^3 z + 2 x^2 y z - x y^3 + 2 x y^2 z - 2 x y z^2 + y z^3=0$ \\
${}^{\blacklozenge}C^A_{{49,A}_{\{ 14 \}}}$ & $x^3 z - x y^3 + y z^3=0$ \\
${}^{\blacklozenge}C^A_{{56,A}_{\{ 0, 1, 0 \}}}$ & $x^3 z + 2 x^2 y z - x^2 z^2 - x y^3 - 2 x y^2 z - 6 x y z^2 + 2 y^3 z + 2 y^2 z^2 + 4 y z^3=0$ \\
$C^{A,B,C}_{57}$ & $2 x^3 z - 2 x^2 y^2 + 5 x^2 z^2 - 16 x y^2 z - 8 x y z^2 + 2 x z^3 + 3 y^4 + 8 y^3 z - 6 y^2 z^2 - 4 y z^3=0$\\
$C^{A,B}_{65}$ & $2 x^3 z - 2 x^2 y^2 - 7 x^2 z^2 - 2 x y^3 - 4 x y^2 z + 26 x y z^2 + 30 x z^3 - 3 y^4 - 26 y^3 z - 81 y^2 z^2 - 98 y z^3 - 40 z^4=0$\\
$C^{A,C}_{65}$ & $6 x^3 z - 6 x^2 y^2 - 8 x^2 z^2 - 3 x y^3 + 25 x y^2 z - 13 x y z^2 + 25 x z^3 -11 y^4 + 19 y^3 z - 33 y^2 z^2 + 13 y z^3 - 14 z^4=0$\\
$C^{A,B}_{82}$ & $x^3 z - x^2 y^2 - 2 x^2 z^2 + 4 x y^2 z + 3 x y z^2 + y^3 z - 2 y z^3=0$\\
$C^{A,C}_{91}$ & $x^3 z - x^2 y^2 - x^2 z^2 + x y^3 - x y^2 z + 3 x y z^2 - x z^3 - 2 y^4 + 4 y^3 z - 6 y^2 z^2 + 4 y z^3 - z^4=0$\\
$C^A_{97}$ & $x^3 z - x^2 y^2 - 5 x^2 z^2 + x y^3 + x y^2 z + 3 x y z^2 + 6 x z^3 - 3 y^2 z^2 - y z^3 - 2 z^4=0$\\
${}^{\bigstar}C^{A,B,D}_{99}$ & $x^3 z - 3 x^2 z^2 + 3 x y^2 z - 3 x y z^2 + 9 x z^3 - y^4 - 6 y^2 z^2 + y z^3 - 8 z^4=0$\\
${}^{\blacklozenge}C^B_{109}$ & $x^3 z - 2 x^2 y z - x^2 z^2 - x y^3 + 6 x y^2 z - 6 x y z^2 + 3 x z^3 + y^4 -6 y^3 z + 10 y^2 z^2 - 5 y z^3=0$\\
$C^C_{113}$ & $x^3 z - x^2 y^2 - 4 x^2 z^2 + x y^3 + 2 x y^2 z + 6 x z^3 - y^3 z - 3 y^2 z^2 + y z^3 - 3 z^4=0$\\
$C^{A,B,C}_{118}$ & $x^3 z - x^2 y^2 - x^2 z^2 + 2 x y^2 z + x y z^2 + x z^3 + y^3 z + y^2 z^2 + y z^3 + z^4=0$\\
$C^{B,C}_{123}$ & $x^3 z - x^2 y^2 + x^2 z^2 - x y^3 - 2 x y^2 z + x z^3 - y^4 - y^3 z - y^2 z^2=0$\\
$C^A_{127}$ & $x^3 z - x^2 y^2 - 3 x^2 z^2 + x y^3 - x y z^2 + 4 x z^3 + 2 y^3 z - 3 y^2 z^2 + 3 y z^3 - 2 z^4=0$\\
$C^B_{139}$ & $x^3 z - x^2 y^2 - 2 x^2 z^2 + x y^3 - 2 x y^2 z + 2 x y z^2 + x z^3 + y^4 -2 y^3 z + 4 y^2 z^2 - 3 y z^3=0$\\
$C^{C,D,E}_{141}$ & $x^3 z - x^2 y^2 + x^2 z^2 - x y^3 + x y^2 z + x z^3 - y^4 - y^3 z - y^2 z^2=0$\\
$C^A_{149}$ & $x^3 z - x^2 y^2 - 3 x^2 z^2 + x y^3 + 3 x y^2 z - 2 x y z^2 + 2 x z^3 - y^4 - y^2 z^2 + y z^3=0$\\
${}^{\blacklozenge}C^A_{151}$ & $x^3 z - 2 x^2 y z - 2 x^2 z^2 - x y^3 + 2 x y^2 z + 4 x y z^2 + x z^3 + y^2 z^2 - 3 y z^3 - 2 z^4=0$\\
$C^B_{169}$ & $x^3 z - x^2 y^2 - 3 x^2 z^2 + x y^3 + 2 x y z^2 + x z^3 + y^2 z^2 - 3 y z^3 + z^4=0$\\
${}^{\blacklozenge}C^B_{179}$ & $x^3 z - 2 x^2 y z - 2 x^2 z^2 - x y^3 + 2 x y^2 z + x y z^2 + 2 x z^3 + y^2 z^2 - y z^3 - z^4=0$\\
$C^E_{187}$ & $x^3 z - x^2 y^2 - x^2 z^2 + x y^3 - x y^2 z - x y z^2 + 2 x z^3 + y^3 z - y^2 z^2 + 3 y z^3=0$\\
$C^F_{203}$ & $x^3 z - x^2 y^2 - 3 x^2 z^2 + x y^3 + 3 x y^2 z - 4 x y z^2 + 4 x z^3 - y^4 + 3 y^3 z - 6 y^2 z^2 + 3 y z^3 - 2 z^4=0$\\
$C^B_{217}$ 
& $3 x^3 z - 3 x^2 y^2 - 11 x^2 z^2 - 3 x y^3 + 13 x y^2 z - 2 x y z^2 + 11 x z^3 -2 y^4 - y^3 z - 4 y^2 z^2 + y z^3 - 2 z^4=0$\\
$C^A_{239}$ & $x^3 z - x^2 y^2 - x^2 z^2 + x y^3 - x y^2 z + x z^3 + y^4 - y^3 z + y z^3 - z^4=0$\\
${}^{\blacklozenge}C^E_{243}$ & $ x^3 z - 3 x^2 z^2 - x y^3 + 9 x y z^2 - 6 x z^3 + 2 y^3 z - 9 y^2 z^2 + 9 y z^3 - 2 z^4=0$\\
${}^{\blacklozenge}C^{A,D}_{243}$ & $x^3 z - x y^3 + 6 x z^3 - 4 y^3 z + 7 z^4 =0$\\
$C^B_{295}$ 
& $x^3 z - x^2 y^2 - x^2 z^2 + x y^3 - x y^2 z + 2 x y z^2 - x z^3 - y^3 z + 3 y^2 z^2 - y z^3=0$\\
$C^D_{329}$ 
& $x^3 z - x^2 y^2 + x y^3 + x y z^2 + x z^3 - y^3 z + 2 y z^3 + z^4 =0$\\
${}^{\blacklozenge}C^D_{369}$ & $x^3 z - 2 x^2 z^2 - x y^3 + 6 x y z^2 - 6 x z^3 - 3 y^2 z^2 + 6 y z^3 - z^4=0$\\
${}^{\blacklozenge}C^{B,K}_{459}$ 
& $x^3 z - x^2 z^2 - x y^3 + 5 x y z^2 - x z^3 + y^4 + 2 y^3 z - y^2 z^2 - 2 y z^3=0$\\
$C^D_{475}$ 
& $x^3 z - x^2 y^2 - 5 x^2 z^2 - x y^3 + x y^2 z + 17 x y z^2 + 14 x z^3 - 2 y^4 - 14 y^3 z - 35 y^2 z^2 - 35 y z^3 - 12 z^4=0$\\
${}^{\blacklozenge}C^H_{855}$ & $x^3 z - x^2 z^2 - x y^3 + 3 x y z^2 - 3 x z^3 + 2 y^3 z - 3 y^2 z^2 + 3 y z^3=0$\\
$C^E_{1175}$ & $ xy^3-x^2y^2 + y^4 + x^3z - 2xy^2z - 2y^3z + x^2z^2 + 2xyz^2 + y^2z^2 - xz^3 + yz^3=0$\\
${}^{\blacklozenge}C^O_{1215}$ 
& $x^3 z - x y^3 + 3 x y z^2 + 5 x z^3 - 6 y^2 z^2 - 3 y z^3 + z^4=0$\\
${}^{\blacklozenge}C^{A,L}_{1215}$ 
& $x^3 z - x y^3 + 3 x y z^2 + 5 x z^3 + 3 y^2 z^2 + 6 y z^3 - 8 z^4 =0$\\
${}^{\bigstar}C^{C,D,E}_{1539}$ & $x^3 z - 3 x^2 z^2 + 3 x y^2 z - 3 x y z^2 + 3 x z^3 - y^4 - 2 y^2 z^2 + y z^3 + 2 z^4=0$\\
\end{longtable}
}


\begin{thebibliography}{1}

\bibitem{accola_70}
R.~Accola.
\newblock Two {T}heorems on {R}iemann {S}urfaces with {N}oncyclic
  {A}utomorphism {G}roups.
\newblock {\em Proc. Amer. Math. Soc.}, 25(3):598--602, 1970.

\bibitem{BakGonPoo}
M.~H. Baker, E.~Gonz{\'a}lez-Jim{\'e}nez, J.~Gonz{\'a}lez, and B.~Poonen.
\newblock Finiteness results for modular curves of genus at least 2.
\newblock {\em Amer. J. Math.}, 127(6):1325--1387, 2005.

\bibitem{Magma}
W.~Bosma, J.~Cannon, and C.~Playoust.
\newblock The {M}agma algebra system. {I}. {T}he user language.
\newblock {\em J. Symbolic Comput.}, 24(3-4):235--265, 1997.
\newblock Computational algebra and number theory (London, 1993). Available on
  \url{http://magma.maths.usyd.edu.au/magma/}.

\bibitem{MR1839918}
C.~Breuil, B.~Conrad, F.~Diamond, and R.~Taylor.
\newblock On the modularity of elliptic curves over {$\Q$}: wild 3-adic
  exercises.
\newblock {\em J. Amer. Math. Soc.}, 14(4):843--939, 2001.

\bibitem{brieskorn} 
E. ~Brieskorn, H.~Kn\"orrer.
\newblock Plane algebraic curves.
\newblock Birkh\"auser, Basel 1986.

\bibitem{cremona97}
J.~E. Cremona.
\newblock Algorithms for modular elliptic curves.
\newblock Cambridge University Press, Cambridge, second edition, 1997.

\bibitem{cremonaweb}
J.~E. Cremona.
\newblock Elliptic curve data.
\newblock Available on \url{http://www.warwick.ac.uk/~masgaj/ftp/data/}, 2006.

\bibitem{Dixmier}
J.~Dixmier.
\newblock {O}n the projective {I}nvariants of quartic plane curves.
\newblock {\em Advances in Math.}, 64:279--304, 1987.

\bibitem{farkas_kra}
H.~M. Farkas and I.~Kra.
\newblock {\em Riemann surfaces}, volume~71 of {\em Graduate Texts in
  Mathematics}.
\newblock Springer-Verlag, New York, second edition, 1992.

\bibitem{MR1692024}
M.~Furumoto and Y.~Hasegawa.
\newblock Hyperelliptic quotients of modular curves {$X\sb 0(N)$}.
\newblock {\em Tokyo J. Math.}, 22(1):105--125, 1999.


\bibitem{GirardKohel}
M.~Girard and D.R. Kohel.
\newblock Classification of genus 3 curves in special strata of the moduli space.
\newblock In {\em Algorithmic Number Theory Symposium - ANTS VIII)}, volume 4076 of {\em Lecture Notes in Comput. Sci.}, pages 346--360. Springer, Berlin, 2006.


\bibitem{J_Gon}
J.~Gonz{\'a}lez.
\newblock {E}quations of hyperelliptic modular curves.
\newblock {\em Ann. Inst. Fourier (Grenoble)}, 41(4):779--795, 1991.

\bibitem{GonGon}
E.~Gonz{\'a}lez-Jim{\'e}nez and J.~Gonz{\'a}lez.
\newblock {M}odular curves of genus 2.
\newblock {\em Math. comp.}, 72:397--418, 2003.

\bibitem{MR1472817}
Y.~Hasegawa.
\newblock Hyperelliptic modular curves {$X\sp *\sb 0(N)$}.
\newblock {\em Acta Arith.}, 81(4):369--385, 1997.

\bibitem{hasegawa:hashimoto96}
Y.~Hasegawa and K.~Hashimoto.
\newblock Hyperelliptic modular curves ${X}\sp *\sb 0({N})$ with square-free
  levels.
\newblock {\em Acta Arith.}, 77(2):179--193, 1996.

\bibitem{komiya_kuri}
A.~Kuribayashi and K.~Komiya.
\newblock On {W}eierstrass points and automorphisms of curves of genus three.
\newblock In {\em Algebraic geometry (Proc. Summer Meeting, Univ. Copenhagen,
  Copenhagen, 1978)}, volume 732 of {\em Lecture Notes in Math.}, pages
  253--299. Springer, Berlin, 1979.

\bibitem{murabayashi92}
N.~Murabayashi.
\newblock On normal forms of modular curves of genus $2$.
\newblock {\em Osaka J. Math.}, 29(2):405--418, 1992.

\bibitem{oyono}
R.~Oyono.
\newblock Non-hyperelliptic modular Jacobians of dimension 3.
\newblock {\em Math. comp.}, 78: 1173--1191, 2009.

\bibitem{Shimura71}
G.~Shimura.
\newblock On elliptic curves with complex multiplication as factors of the Jacobians of modular function fields.
\newblock {\em Nagoya Math. J.}, 43, 199--208, 1971. 

\bibitem{Shimura73}
G.~Shimura.
\newblock On the factors of the jacobian variety of a modular function field.
\newblock {\em J. Math. Soc. Japan}, 25(3):523--544, 1973.

\bibitem{shimura-taniyama}
G.~Shimura and Y.~Taniyama.
\newblock  Complex Multiplication of Abelian Varieties and its Applications to Number Theory.
\newblock  Vol. 6 (Pub. Japan Math. Soc., 1961).

\bibitem{MShimura95}
M.~Shimura.
\newblock {D}efining {E}quations of {M}odular {C}urves {$X_0 (N)$}.
\newblock {\em Tokyo J. Math.}, 18(2):443--456, 1995.

\bibitem{TWiles}
A.~R. Taylor and A.~Wiles.
\newblock {R}ing-theoretic properties of certain {H}ecke algebras.
\newblock {\em Ann. of Math. (2)}, 141(3):553--572, 1995.

\bibitem{Wiles}
A.~Wiles.
\newblock {M}odular elliptic curves and {F}ermat's last theorem.
\newblock {\em Ann. of Math. (2)}, 141(3):443--551, 1995.

\end{thebibliography}
\end{document}